\documentclass[aop]{imsart}
\usepackage{helvet}

\usepackage{amsmath}
\usepackage{amssymb}
\usepackage{amsthm}

\newtheorem{Def}{Definition}
\newtheorem{Prop}{Proposition}
\newtheorem{Thm}{Theorem}

\newtheorem{Cor}{Corollary}

\newenvironment{keywords}
{ \textbf{Keywords:} \begin{itshape} }
{\end{itshape} }

\begin{document}


\begin{frontmatter}

\title{Stable adiabatic times for a continuous evolution of Markov chains}

\begin{aug}

\author{\fnms{Kyle}  \snm{Bradford}\ead[label=e1]{kyle.bradford@gmail.com}}

\affiliation{University of Nevada, Reno} 
\address{Department of Mathematics and Statistics \\ 
Davidson Mathematics and Science Building, Room 218 \\
1664 N. Virginia Street
Reno, NV 89557-0084 á 775-784-6773, USA}

\end{aug}

\begin{abstract}
\label{sec: abstract}
This paper continues the discussion on the stability of time-inhomogeneous Markov chains.  In particular, this paper defines a time-inhomogeneous, discrete-time Markov chain governed by a continuous evolution in the appropriate martrix space.  This matrix space, $\mathcal{P}_{n}^{ia}$,  is the space of all stochastic matrices that are irreducible and aperiodic.  For this new type of evolution there is a definition of a specific type of stability called the stable adiabatic time.  This measure is bounded by a function of the optimal mixing time over the evolution.  Namely, for a time-inhomogeneous, discrete-time Markov chain governed by a continuous evolution through a function $\mathbf{P}: [0,1] \rightarrow \mathcal{P}_{n}^{ia}$  and $0 < \epsilon < \frac{1}{2 \sqrt{n}}$  $$t_{sad}(\mathbf{P}, \epsilon) \leq \frac{3n^{3 \slash 2} L t_{mix}^{2}(\mathbf{P}_{\infty}, \epsilon)}{(1-2\sqrt{n} \epsilon) \epsilon}$$

\noindent where $L$  is a Lipschitz constant related to the function $\mathbf{P}$.
\end{abstract}

\end{frontmatter}
\thispagestyle{empty}


\section{\sc {Introduction}}
\label{sec: introduction}
The stability of Markov chains is relevant to many applications in math and science.  For background literature on Markov chains one can reference \cite{isaacson1976markov, karlin1975first, levin2009markov}.  The first type of stability often encountered in a discussion of time-homogeneous, irreducible and aperiodic Markov chains is the mixing time \cite{aldous2002reversible, levin2009markov}.  $\| \cdot \|_{TV}$  is reserved as the total variation norm and $\| \cdot \|_{k}$  as the $\ell^{k}( \mathbb{R}^{n})$  norm.

\begin{Def}\label{def: one}
For $\epsilon > 0$  the \underline{mixing time}  of a time-homogeneous, irreducible and aperiodic Markov chain governed by a probability transition matrix $\mathbf{P}$, which has unique stationary distribution $\mathbf{\pi}$,  is defined as:
\begin{equation} \label{eq: one}
t_{mix}(\mathbf{P}, \epsilon) = \inf \{ T \in \mathbb{N} : \| \mathbf{\nu} \mathbf{P}^{T} - \mathbf{\pi} \|_{TV} \leq \epsilon \}
 \end{equation}

\noindent over all distributions $\mathbf{\nu}$. 
\end{Def}

There are many examples of applications of the mixing time \cite{ross2006simulation}.  Although time-homogeneous Markov chains have been thoroughly studied, the stability of time-inhomogeneous Markov chains is much less attainable.  People have attempted to discuss a related mixing time for time-inhomogeneous Markov chains \cite{saloff2006convergence, saloff2009merging, saloff2011merging}.  Although these attempts have been for general time-inhomogeneous Markov chains, some time-inhomogeneous Markov chains with many real world applications have been studied.  The types of time-inhomogeneous Markov chains discussed in this paper are best described as adiabatic.

This article continues the effort in \cite{bradford2011adiabatic, bradford2012stable, kovchegov2010note} to bound the stable adiabatic time of an evolving, time-inhomogeneous Markov chain by a function of the largest mixing time over the entire evolution.  Specifically this paper makes three important contributions:  1) finding an exact bound rather than an asymptotic bound, 2) finding a tighter, optimal bound of the stable adiabatic time and 3) expanding the types of evolutions to include all continuous transitions in the appropriate matrix space.  Some of the strongest applications of the adiabatic time and the stable adiabatic time come from quantum physics and quantum computation.  Namely, the quantum adiabatic theorem from physics \cite{fock2004selected, kato1950on} and quantum adiabatic computing \cite{krovi2010adiabatic}.  There is a strong presence of adiabatic processes in optimization algorithms in queueing systems \cite{duong2014network}, network design \cite{rajagopalan2009network}  and network performance \cite{zacharias2011adiabatic}.  There is also an application to the stability of an Ising model with Glauber dynamics \cite{bradford2011adiabatic}.  Many of these applications were discussed in detail in previous works.  For example, the quantum adiabatic theorem was discussed in detail in \cite{ambainis2006elementary, bradford2011adiabatic, bradford2012stable, kovchegov2010note}  and the quantum computation applications were discussed in \cite{bradford2012stable}.  

In \cite{bradford2012stable} the time-inhomogeneous Markov chain was specifically governed by a convex-combination evolution of two irreducible, aperiodic probability transition matrices.  In particular there were matrices $\mathbf{P}_{0}$  and $\mathbf{P}_{1}$  and  $\mathbf{P}_{t} = (1-t) \mathbf{P}_{0} + t \mathbf{P}_{1}$.  Given a large integer $T$  the probability transition matrix at time $k \leq T$  for the time-inhomogeneous Markov chain was $\mathbf{P}_{\frac{k}{T}}$.  Naturally if stochastic matrices $\mathbf{P}_{0}$  and $\mathbf{P}_{1}$  are both irreducible and aperiodic, then $\mathbf{P_{s}}$ is both irreducible and aperiodic for $s \in [0,1]$.  This allows for a definition of the mixing time for each $s \in [0,1]$.  Taking the supremum of all of these mixing times is one of the ways that one can discuss stability for the time-inhomogeneous Markov chains with probability transition matrices $\mathbf{P_{\frac{k}{T}}}$.  The following definition makes this formal.

\begin{Def} \label{def: two}
For $\epsilon > 0$  the \underline{largest mixing time} of a time-inhomogeneous, discrete-time Markov chain governed by \underline{a convex-combination evolution} between the irreducible and aperiodic $\mathbf{P_{0}}$  and $\mathbf{P_{1}}$ 
\begin{equation} \label{eq: two}
t_{mix}(\mathbf{P}_{0}, \mathbf{P}_{1}, \epsilon) =\sup_{s \in [0,1]} \{ t_{mix}(\mathbf{P}_{s},\epsilon) \}.
\end{equation}
\end{Def}

This paper has already mentioned the stable adiabatic time a few times without giving the formal definition.  Now there is enough background information to make this definition for convex-combination evolutions.  This was the main object of study in \cite{bradford2012stable} and will motivate the analogue that we will use in this paper.

\begin{Def} \label{def: three}
For $\epsilon > 0$ the \underline{stable adiabatic time} of a time-inhomogeneous, discrete-time Markov chain governed by \underline{a convex-combination evolution} between the irreducible and aperiodic $\mathbf{P_{0}}$  and $\mathbf{P_{1}}$, which has unique stationary distribution  $\mathbf{\pi_{\frac{k}{T}}}$  for the probability transition matrix $\mathbf{P_{\frac{k}{T}}}$, is defined as :
\begin{equation} \label{eq: three}
t_{sad}(\mathbf{P_{0}}, \mathbf{P_{1}}, \epsilon) = \inf \{ T \in \mathbb{N} : \| \mathbf{\pi_{0}} \mathbf{P_{\frac{1}{T}}} \cdots \mathbf{P_{\frac{k}{T}}} - \mathbf{\pi_{\frac{k}{T}}} \|_{TV} < \epsilon \text{ for } 1 \leq k \leq T \}.
\end{equation}
\end{Def}

The stable adiabatic time is another type of stability for these types of time-inhomogeneous Markov chains.  It is natural to ask how the two previous definitions compare.  This was discussed in \cite{bradford2012stable} for these specific convex-combination evolutions.  The following asymptotic result was discovered in \cite{bradford2012stable} relating the stable adiabatic time and the largest mixing time.

\begin{Thm} \label{thm: one}
Given a time-inhomogeneous, discrete-time Markov chain governed by a convex-combination evolution between the irreducible and aperiodic $\mathbf{P_{0}}$  and $\mathbf{P_{1}}$, for any $\epsilon > 0$,

\begin{equation} \label{eq: four}
t_{sad}(\mathbf{P_{0}}, \mathbf{P_{1}}, \epsilon) = O \left( \frac{t_{mix}^{4}(\mathbf{P}_{0}, \mathbf{P}_{1}, \epsilon \slash 2)}{\epsilon^{3}} \right).
\end{equation}
\end{Thm}

The main goal of this paper is to expand the types of evolutions that can take place.  To elaborate, first let $\mathcal{M}_{n}([0,1])$  be the collection of all $n \times n$ matrices with entries in $[0,1]$.  Define $\mathcal{P}_{n} = \{ \mathbf{P} \in \mathcal{M}_{n}([0,1]) : \mathbf{P} \mathbf{1} = \mathbf{1} \}$  where $\mathbf{1}$  is the n dimensional column vector with all entries $1$  and define $$\mathcal{P}_{n}^{ia} = \{ \mathbf{P} \in \mathcal{P}_{n} : \mathcal{P} \text{ is irreducible and aperiodic} \}.$$

To describe continuity in this matrix space the standard matrix norm will be used.  Specifically for a matrix $\mathbf{M}$  the matrix norm is defined as $\| \mathbf{M} \| =\max_{\nu} \| \nu M \|_{1}$  where the maximum is taken over all probability distributions $\nu$.  This paper considers continuous functions $\mathbf{P}: [0,1] \rightarrow \mathcal{P}_{n}^{ia}$  with respect to the Matrix norm to build the more general types of evolutions.  One can now allow the time-inhomogeneous Markov chains to be governed by a continuous evolution defined through the function $\mathbf{P}$.  Given a large integer $T$  the probability transition matrix at time $k \leq T$  for the time-inhomogeneous Markov chain was $\mathbf{P} \left(\frac{k}{T} \right)$.  Because all probability transition matrices are in $\mathcal{P}_{n}^{ia}$  a mixing time exists for all $s \in [0,1]$.  The supremum can be taken again to make a metric for stability for time-inhomogeneous Markov chains governed by these continuous evolutions.  Note the difference between this definition and Definition \ref{def: two}.

\begin{Def} \label{def: four}
For $\epsilon > 0$  the \underline{largest mixing time} of a time-inhomogeneous, discrete-time Markov chain governed by \underline{a continuous evolution in $\mathcal{P}_{n}^{ia}$}, written as $t_{mix}(\mathbf{P}_{\infty}, \epsilon)$, is defined as follows:
\begin{equation} \label{eq: five}
t_{mix}(\mathbf{P}_{\infty}, \epsilon) =\sup_{s \in [0,1]} \{ t_{mix}(\mathbf{P}(s),\epsilon) \}.
\end{equation}
\end{Def}

Finally the version of the stable adiabatic time used in this paper can be introduced.  The key difference we Definition \ref{def: three}  is the type of evolution.  This version of the stable adiabatic time allows for a more general, continuous evolution in $\mathcal{P}_{n}^{ia}$.

\begin{Def} \label{def: five}
For $\epsilon > 0$ the \underline{stable adiabatic time}  of a time-inhomogeneous, discrete-time Markov chain governed by \underline{a continuous evolution in $\mathcal{P}_{n}^{ia}$}, written as $t_{sad}(\mathbf{P}, \epsilon)$, is defined as follows:
\begin{equation} \label{eq: six}
\begin{split}
t_{sad}(\mathbf{P}, \epsilon) = \inf \Bigg\{ T \in \mathbb{N} : \Bigg\| \mathbf{\pi} \left( 0 \right) \mathbf{P} \left( \frac{1}{T} \right) \cdots \mathbf{P}\left( \frac{k}{T} \right) - \mathbf{\pi} &\left( \frac{k}{T} \right) \Bigg\|_{TV} < \epsilon \\
&\text{ for } 1 \leq k \leq T \}.
\end{split}
\end{equation}
\end{Def}

With all of these definitions formally laid out it can be said that the purpose of this paper is to find a relationship between $t_{sad}(\mathbf{P}, \epsilon)$  and $t_{mix}(\mathbf{P}_{\infty}, \epsilon)$  in an analogous way as Theorem \ref{thm: one}.  The machinery in this paper allows for a better result and derives an optimal result.  The rest of the paper is organized as follows: Section \ref{sec: supporting} introduces the necessary background information to allow for a succinct proof of the main result, Section \ref{sec: argument} gives the main result of the paper and gives a detailed proof of the main result, Section \ref{sec: conclusion} gives a context of the importance of the result and additional proofs and argumentation is outlined in Section \ref{sec: proofs}.

\section{\sc {Supporting Material}
\label{sec: supporting}}
Whenever one wants to make a proof about continuous function spaces, a common proof technique involves using a dense subset known as the Lipschitz continuous function with finite Lipschitz constant.  This section introduces two important propositions that aide the proof of the main result.  For both the first proposition and the main result in Section \ref{sec: argument} using Lipschitz continuous functions allows for a keen insight as to what commands these time-inhomogeneous Markov chains governed by a continuous evolution.  In this matrix space the following definition of a Lipschitz continuous function is used.

\begin{Def} \label{def: six}
A function $\mathbf{P^{*}}: [0,1] \rightarrow \mathcal{M}_{n}([0,1])$  is \underline{Lipschitz}  if there exists a positive constant $L$, called the \underline{Lipschitz constant}, so that 
\begin{equation} \label{eq: seven}
\| \mathbf{P^{*}} \left( x \right) - \mathbf{P^{*}} \left( y \right) \| \leq L \big| x - y \big|
\end{equation}
for $x,y \in [0,1]$.
\end{Def}

The function $\mathbf{P}: [0,1] \rightarrow \mathcal{P}_{n}^{ia}$  creates a function $\pi: [0,1] \rightarrow \mathbb{R}^{n}$.  By definition $\mathbf{P}$  is continuous with respect to the matrix norm, so a natural question is whether $\pi$  is a continuous function with respect to the total variation norm.  The following proposition declares that it is.  This in and of itself is not that surprising, but the nature of how it is continuous gives information that will be necessary in proving the main result.

\begin{Prop} \label{pr: one}
Let $\sigma = \inf_{s \in [0,1]} \{ \sigma(s) \}$  where $\sigma(s)$  is the smallest nonzero singular value of $\mathbb{I} - \mathbf{P}(s)$. \\

\noindent If $\mathbf{P}:[0,1] \rightarrow \mathcal{P}_{n}^{ia}$  is a continuous function with respect to the matrix norm,  then $\mathbf{\pi}:[0,1] \rightarrow \mathbb{R}^{n}$  is uniformly continuous with respect to the total variation norm.  In particular, for $\epsilon > 0$  there exists a positive constant $L$  such that for $s \in [0,1]$  and
\begin{equation} \label{eq: eight}
\delta = \frac{ \epsilon \sigma}{3 L n^{3 \slash 2}},
\end{equation}

\noindent $t \in \{ [0,1]: \big| t - s \big| \leq \delta \}$  implies that $\| \mathbf{\pi}(t) - \mathbf{\pi}(s) \|_{TV} \leq \epsilon$.
\end{Prop}

Notice that in the above proposition the continuity depends on the smallest nonzero singular value of the function $\mathbf{P}$  throughout the entire evolution.  This value $\sigma$  has information relating to the largest mixing time of $\mathbf{P}$  throughout the entire evolution.  The following proposition makes this point.

\begin{Prop} \label{pr: two}
Let $\mathbf{P}:[0,1] \rightarrow \mathcal{P}_{n}^{ia}$  be a continuous function with respect to the matrix norm.  Let $\sigma = \inf_{s \in [0,1]} \{ \sigma(s) \}$  where $\sigma(s)$  is the smallest nonzero singular value of $\mathbb{I} - \mathbf{P}(s)$. \\

Given $\epsilon > 0$, 
\begin{equation} \label{eq: nine}
\frac{1-2\sqrt{n}\epsilon}{\sigma} \leq t_{mix}(\mathbf{P}_{\infty}, \epsilon).
\end{equation}
\end{Prop}

Instead of including a proof of Proposition \ref{pr: two} note that the proof falls rather directly from a similar argument in \cite{bradford2012stable}.  In this paper one can find a similar relationship between the smallest nonzero singular value of a matrix and its mixing time.  Here the only thing to note is that the mixing time of time-homogeneous Markov chain associated with the smallest nonzero singular value is smaller that the supremum of all mixing times throughout the entire evolution.

This provides all the necessary background to approach our main result.  This result is now addressed in Section \ref{sec: argument}.

\section{\sc {Main Result}
\label{sec: argument}}
The main result of this paper is given in the following theorem and proven in this section.  It will provide the necessary analogue for the bound on the stable adiabatic time for time-inhomogeneous Markov chains governed by a continuous evolution by a function of the largest mixing time over the entire evolution.  Note that this result differs from Theorem \ref{thm: one} by not being an asymptotic result and having a lower power of the largest mixing time bound the stable adiabatic time.  After this theorem is proven the impact of the result will be discussed in Section \ref{sec: conclusion}.  

\begin{Thm} \label{thm: two}
Given a time-inhomogeneous, discrete-time Markov chain governed by a continuous evolution in $\mathcal{P}_{n}^{ia}$, for $0 < \epsilon < \frac{1}{2 \sqrt{n}}$  and $\mathbf{P}:[0,1] \rightarrow \mathcal{P}_{n}^{ia}$  a continuous function with respect to the matrix norm we have that 
\begin{equation} \label{eq: ten}
t_{sad}(\mathbf{P}, \epsilon) \leq \frac{3n^{3 \slash 2}Lt_{mix}^2(\mathbf{P}_{\infty},\epsilon)}{(1 - 2\sqrt{n} \epsilon)\epsilon}
\end{equation}
\end{Thm}

\begin{proof}
\ 

\noindent For the proof first let $0 < \epsilon < \frac{1}{2\sqrt{n}}$  and let $\mathbf{P}:[0,1] \rightarrow \mathcal{P}_{n}^{ia}$  be a continuous function with respect to the matrix norm.  It is important to remember that $t_{mix}(\mathbf{P}_{\infty}, \epsilon \slash 2)$  exists and is a natural number. \\

\noindent Recall that the space of Lipschitz continuous functions from $[0,1]$  to $\mathcal{P}_{n}^{ia}$  with finite Lipschitz constant is dense in the space of continuous functions from $[0,1]$  to $\mathcal{P}_{n}^{ia}$.  This implies that one can find a Lipschitz continuous function $\mathbf{P^{*}}:[0,1] \rightarrow \mathcal{P}_{n}^{ia}$  with Lipschitz constant $L$  such that $$ \| \mathbf{P}(t) - \mathbf{P^{*}}(t) \| \leq \frac{ \epsilon}{4 t_{mix} \left(\mathbf{P}_{\infty}, \epsilon \slash 2 \right)} $$

\noindent for all $t \in [0,1]$. \\

\noindent The goal of this proof is to select a value of $T$  large enough so that $$\Bigg\| \mathbf{\pi}(0) \mathbf{P} \left( \frac{1}{T} \right) \mathbf{P} \left( \frac{2}{T} \right) \cdots \mathbf{P} \left( \frac{k-1}{T} \right) \mathbf{P} \left( \frac{k}{T} \right) - \mathbf{\pi}\left( \frac{k}{T} \right) \Bigg\|_{TV} \leq \epsilon$$

\noindent for $1 \leq k \leq T$. \\

\noindent Let $$T = \frac{3 n^{3 \slash 2} L t_{mix}^{2} \left(\mathbf{P}_{\infty}, \epsilon \slash 2 \right)}{ (1 - 2 \sqrt{n} \epsilon)\epsilon}.$$
\ 

\noindent At this point the proof is decomposed into two parts.  \\

\noindent \underline{Part 1. Assume that $k \geq t_{mix}(\mathbf{P}_{\infty}, \epsilon \slash 2)$} \\

\noindent Let $N=k - t_{mix}(\mathbf{P}_{\infty}, \epsilon \slash 2)$. \\

\noindent Observe that 
\begin{align*} 
\mathbf{\pi}(0) \mathbf{P} \left( \frac{1}{T} \right) &\mathbf{P} \left( \frac{2}{T} \right) \cdots \mathbf{P} \left( \frac{k-1}{T} \right) \mathbf{P} \left( \frac{k}{T} \right) \\
&= \mathbf{\nu}_{N} \left( \mathbf{P} \left( \frac{k}{T} \right) + \left( \mathbf{P} \left( \frac{N+1}{T} \right) - \mathbf{P} \left( \frac{k}{T} \right) \right) \right) \mathbf{P}^{\circ}_{N+2} \\
&= \mathbf{\nu}_{N} \mathbf{P} \left( \frac{k}{T} \right) \mathbf{P}^{\circ}_{N+2} + \mathbf{\nu}_{N} \left( \mathbf{P} \left( \frac{N+1}{T} \right) - \mathbf{P} \left( \frac{k}{T} \right) \right) \mathbf{P}^{\circ}_{N+2}.
\end{align*}

\noindent where $\mathbf{\nu}_{N} = \mathbf{\pi}(0) \mathbf{P} \left( \frac{1}{T} \right) \mathbf{P} \left( \frac{2}{T} \right) \cdots \mathbf{P} \left( \frac{N}{T} \right)$, $\mathbf{P}^{\circ}_{\ell} =  \mathbf{P} \left( \frac{\ell}{T} \right) \cdots \mathbf{P}\left( \frac{k}{T} \right)$.  \\

\noindent By continuing this process for $\mathbf{P} \left( \frac{i}{T} \right)$  for $i \geq N+2$, it can be shown that 
\begin{align*} 
\mathbf{\pi}(0) \mathbf{P} &\left( \frac{1}{T} \right) \mathbf{P} \left( \frac{2}{T} \right) \cdots \mathbf{P} \left( \frac{k-1}{T} \right) \mathbf{P} \left( \frac{k}{T} \right) \\
&= \mathbf{\nu}_{N} \left( \mathbf{P} \left( \frac{k}{T} \right) \right)^{k-N} \\
&+ \sum_{\ell=0}^{k-N - 2} \mathbf{\nu}_{N} \left( \mathbf{P} \left( \frac{k}{T} \right) \right)^{\ell} \left( \mathbf{P} \left( \frac{N+1+\ell}{T} \right) - \mathbf{P} \left( \frac{k}{T} \right) \right) \mathbf{P}^{\circ}_{N+2+\ell}.
\end{align*}

\noindent By the triangle inequality, it can be shown that
\begin{align*} 
\Bigg\| \mathbf{\pi}(0) \mathbf{P} &\left( \frac{1}{T} \right) \mathbf{P} \left( \frac{2}{T} \right) \cdots \mathbf{P} \left( \frac{k-1}{T} \right) \mathbf{P} \left( \frac{k}{T} \right) - \mathbf{\pi} \left( \frac{k}{T} \right) \Bigg\|_{TV} \\
&\leq  \Bigg\| \mathbf{\nu}_{N} \left( \mathbf{P}\left( \frac{k}{T} \right) \right)^{k-N} - \mathbf{\pi}\left( \frac{k}{T} \right) \Bigg\|_{TV} \\
&+ \sum_{\ell=0}^{k-N-2} \Bigg\| \mathbf{\nu}_{N} \left( \mathbf{P} \left( \frac{k}{T} \right) \right)^{\ell} \left( \mathbf{P} \left( \frac{N+1+\ell}{T} \right) - \mathbf{P} \left( \frac{k}{T} \right) \right) \mathbf{P}^{\circ}_{N+2+\ell} \Bigg\|_{TV}.
\end{align*}

\noindent Because $2 \| \mu - \nu \|_{TV} = \| \mu - \nu \|_{1}$  whenever $\mu$  and $\nu$  is a probability distribution,  
\begin{align*} 
\Bigg\| \mathbf{\pi}(0) \mathbf{P} &\left( \frac{1}{T} \right) \mathbf{P} \left( \frac{2}{T} \right) \cdots \mathbf{P} \left( \frac{k-1}{T} \right) \mathbf{P} \left( \frac{k}{T} \right) - \mathbf{\pi}\left( \frac{k}{T} \right) \Bigg\|_{TV} \\
&\leq  \Bigg\| \mathbf{\nu}_{N} \left( \mathbf{P}\left( \frac{k}{T} \right) \right)^{k-N} - \mathbf{\pi} \left( \frac{k}{T} \right) \Bigg\|_{TV} + \frac{1}{2} \sum_{\ell=0}^{k-N-2} \| \mathbf{\nu}^{\circ}_{\ell} \mathbf{P}^{\circ}_{N+2+\ell} \|_{1} 
\end{align*}

\noindent where $\mathbf{\nu}^{\circ}_{\ell} = \mathbf{\nu}_{N} \left( \mathbf{P} \left( \frac{k}{T} \right) \right)^{\ell} \left( \mathbf{P} \left( \frac{N+1+\ell}{T} \right) - \mathbf{P} \left( \frac{k}{T} \right) \right).$ \\

\noindent Notice that for $0 \leq \ell \leq k-N-2$, $\mathbf{P}^{\circ}_{N+2+\ell}$  is a probability transition matrix.  This will imply that
\begin{align*}
\| \mathbf{\nu}^{\circ}_{\ell} \mathbf{P}^{\circ}_{N+2+\ell} \|_{1} &= \sum_{j=1}^{n} \big| \sum_{i=1}^{n} \mathbf{\nu}^{\circ}_{\ell} \left( i \right) \mathbf{P}^{\circ}_{N+2+\ell} \left( i,j \right) \big| \\
&\leq \sum_{j=1}^{n}  \sum_{i=1}^{n} \big| \mathbf{\nu}^{\circ}_{\ell} \left( i \right) \big| \mathbf{P}^{\circ}_{N+2+\ell} \left( i,j \right) \\
&= \sum_{i=1}^{n} \big| \mathbf{\nu}^{\circ}_{\ell} \left( i \right) \big| \sum_{j=1}^{n}  \mathbf{P}^{\circ}_{N+2+\ell} \left( i,j \right) \\
&= \sum_{i=1}^{n} \big| \mathbf{\nu}^{\circ}_{\ell} \left( i \right) \big| \\
&= \| \mathbf{\nu}^{\circ}_{\ell} \|_{1}.
\end{align*}

\noindent Therefore
\begin{align*} 
\Bigg\| \mathbf{\pi}(0) \mathbf{P} &\left( \frac{1}{T} \right) \mathbf{P} \left( \frac{2}{T} \right) \cdots \mathbf{P} \left( \frac{k-1}{T} \right) \mathbf{P} \left( \frac{k}{T} \right) - \mathbf{\pi}\left( \frac{k}{T} \right) \Bigg\|_{TV} \\
&\leq  \Bigg\| \mathbf{\nu}_{N} \left( \mathbf{P}\left( \frac{k}{T} \right) \right)^{k-N} - \mathbf{\pi} \left( \frac{k}{T} \right) \Bigg\|_{TV} \\
&+ \frac{1}{2} \sum_{\ell=0}^{k-N-2} \Bigg\| \mathbf{\nu}_{N} \left( \mathbf{P} \left( \frac{k}{T} \right) \right)^{\ell} \left( \mathbf{P} \left( \frac{N+1+\ell}{T} \right) - \mathbf{P} \left( \frac{k}{T} \right) \right) \Bigg\|_{1}. 
\end{align*}

\noindent It is clear that $\mathbf{\nu}_{N} \left( \mathbf{P} \left( \frac{k}{T} \right) \right)^{\ell}$  is a probability vector for $0 \leq \ell \leq k-N-2$, so naturally 
\begin{align*} 
\Bigg\| \mathbf{\pi}(0) \mathbf{P} \left( \frac{1}{T} \right) &\mathbf{P} \left( \frac{2}{T} \right) \cdots \mathbf{P} \left( \frac{k-1}{T} \right) \mathbf{P} \left( \frac{k}{T} \right) - \mathbf{\pi}\left( \frac{k}{T} \right) \Bigg\|_{TV} \\
&\leq  \max_{\nu} \Bigg\| \mathbf{\nu} \left( \mathbf{P}\left( \frac{k}{T} \right)  \right)^{k-N} - \mathbf{\pi} \left( \frac{k}{T} \right) \Bigg\|_{TV} \\
&+ \frac{1}{2} \sum_{\ell=0}^{k-N-2} \max_{\nu} \Bigg\| \mathbf{\nu} \left( \mathbf{P} \left( \frac{N+1+\ell}{T} \right) - \mathbf{P} \left( \frac{k}{T} \right) \right) \Bigg\|_{1} 
\end{align*}

\noindent where the maximum is taken over all probability vectors $\mathbf{\nu}$. \\

\noindent Because $k-N=t_{mix}(\mathbf{P}_{\infty}, \epsilon \slash 2) \geq t_{mix}(\mathbf{P}\left( \frac{k}{T} \right), \epsilon \slash 2)$,  it is easy to see that
\begin{align*} 
\Bigg\| \mathbf{\pi}(0) \mathbf{P} \left( \frac{1}{T} \right) &\mathbf{P} \left( \frac{2}{T} \right) \cdots \mathbf{P} \left( \frac{k-1}{T} \right) \mathbf{P} \left( \frac{k}{T} \right) - \mathbf{\pi}\left( \frac{k}{T} \right) \Bigg\|_{TV} \\
&\leq  \frac{\epsilon}{2} + \frac{1}{2} \sum_{\ell=0}^{k-N-2} \max_{\nu} \Bigg\| \mathbf{\nu} \left( \mathbf{P} \left( \frac{N+1+\ell}{T} \right) - \mathbf{P} \left( \frac{k}{T} \right) \right) \Bigg\|_{1}.
\end{align*}

\noindent Observe that the terms in the sum of the right hand side of the inequality are now the matrix norms for the matrices $\mathbf{P} \left( \frac{N+1+\ell}{T} \right) - \mathbf{P} \left( \frac{k}{T} \right)$  for $0 \leq \ell \leq k-N-2$.  This would imply that

\begin{align*} 
\Bigg\| \mathbf{\pi}(0) \mathbf{P} \left( \frac{1}{T} \right) &\mathbf{P} \left( \frac{2}{T} \right) \cdots \mathbf{P} \left( \frac{k-1}{T} \right) \mathbf{P} \left( \frac{k}{T} \right) - \mathbf{\pi}\left( \frac{k}{T} \right) \Bigg\|_{TV} \\
&\leq \frac{\epsilon}{2} + \frac{1}{2} \sum_{\ell=0}^{k-N-2} \Bigg\| \mathbf{P} \left( \frac{N+1+\ell}{T} \right) - \mathbf{P} \left( \frac{k}{T} \right) \Bigg\|. 
\end{align*}

\noindent By adding and subtracting the same value to the above inequality and then using the triangle inequality 
\begin{align*} 
\Bigg\| \mathbf{\pi}(0) \mathbf{P} \left( \frac{1}{T} \right) &\mathbf{P} \left( \frac{2}{T} \right) \cdots \mathbf{P} \left( \frac{k-1}{T} \right) \mathbf{P} \left( \frac{k}{T} \right) - \mathbf{\pi}\left( \frac{k}{T} \right) \Bigg\|_{TV} \\
&\leq  \frac{\epsilon}{2} + \frac{1}{2} \sum_{\ell=0}^{k-N-2} \Bigg\| \mathbf{P^{*}} \left( \frac{N+1+\ell}{T} \right) - \mathbf{P^{*}} \left( \frac{k}{T} \right) \Bigg\| \\
 &+ \frac{1}{2} \sum_{\ell=0}^{k-N-2} \Bigg\| \mathbf{P} \left( \frac{N+1+\ell}{T} \right) - \mathbf{P^{*}} \left( \frac{N+1+\ell}{T} \right) \Bigg\| \\
 &+ \frac{1}{2} \sum_{\ell=0}^{k-N-2} \Bigg\| \mathbf{P^{*}} \left( \frac{k}{T} \right) - \mathbf{P} \left( \frac{k}{T} \right) \Bigg\|.
\end{align*}

\noindent Using the density of the Lipschitz continuous functions with finite Lipschitz constant in the continuous function space
\begin{align*} 
\Bigg\| \mathbf{\pi}(0) \mathbf{P} \left( \frac{1}{T} \right) &\mathbf{P} \left( \frac{2}{T} \right) \cdots \mathbf{P} \left( \frac{k-1}{T} \right) \mathbf{P} \left( \frac{k}{T} \right) - \mathbf{\pi}\left( \frac{k}{T} \right) \Bigg\|_{TV} \\
&\leq  \frac{\epsilon}{2} + \frac{1}{2} \sum_{\ell=0}^{k-N-2} \Bigg\| \mathbf{P^{*}} \left( \frac{N+1+\ell}{T} \right) - \mathbf{P^{*}} \left( \frac{k}{T} \right) \Bigg\| \\
 &+ \sum_{\ell=0}^{k-N-2} \frac{\epsilon}{4 t_{mix} \left(\mathbf{P}_{\infty}, \epsilon \slash 2 \right)}.
\end{align*}

\noindent Because $\mathbf{P^{*}}: [0,1] \rightarrow \mathcal{P}_{n}^{ia}$  is a Lipschitz continuous function with Lipschitz constant $L$,  it can be shown that 
\begin{align*}
\Bigg\| \mathbf{\pi}(0) \mathbf{P} \left( \frac{1}{T} \right) &\mathbf{P} \left( \frac{2}{T} \right) \cdots \mathbf{P} \left( \frac{k-1}{T} \right) \mathbf{P} \left( \frac{k}{T} \right) - \mathbf{\pi}\left( \frac{k}{T} \right) \Bigg\|_{TV} \\
&\leq  \frac{\epsilon}{2} + \frac{L}{2} \sum_{\ell=0}^{k-N-2} \Bigg| \frac{N+1+\ell}{T} -  \frac{k}{T}  \Bigg| \\
&+ \sum_{\ell=0}^{k - N - 2} \frac{\epsilon}{4 t_{mix} \left(\mathbf{P}_{\infty}, \epsilon \slash 2 \right)}. 
\end{align*}

\noindent After relabeling the sum
\begin{align*}
\Bigg\| \mathbf{\pi}(0) \mathbf{P} \left( \frac{1}{T} \right) &\mathbf{P} \left( \frac{2}{T} \right) \cdots \mathbf{P} \left( \frac{k-1}{T} \right) \mathbf{P} \left( \frac{k}{T} \right) - \mathbf{\pi}\left( \frac{k}{T} \right) \Bigg\|_{TV} \\
&\leq \frac{\epsilon}{2} + \frac{L}{4T} (k - N - 1) (k - N) \\
&+  \frac{\epsilon}{4 t_{mix} \left(\mathbf{P}_{\infty}, \epsilon \slash 2 \right)} (k-N-1).
\end{align*}

\noindent Because $k-N = t_{mix}(\mathbf{P}_{\infty}, \epsilon \slash 2)$  
\begin{align*}
\Bigg\| \mathbf{\pi}(0) \mathbf{P} \left( \frac{1}{T} \right) &\mathbf{P} \left( \frac{2}{T} \right) \cdots \mathbf{P} \left( \frac{k-1}{T} \right) \mathbf{P} \left( \frac{k}{T} \right) - \mathbf{\pi}\left( \frac{k}{T} \right) \Bigg\|_{TV} \\
&\leq \frac{3 \epsilon}{4} + \frac{L}{4T} t_{mix}^{2} \left(\mathbf{P}_{\infty}, \epsilon \slash 2 \right).
\end{align*} 

\noindent $T$  was selected to be large enough.  In fact, $$T = \frac{3 n^{3 \slash 2} L t_{mix}^{2} \left(\mathbf{P}_{\infty}, \epsilon \slash 2 \right)}{ (1 - 2 \sqrt{n} \epsilon) \epsilon} \geq \frac{L t_{mix}^{2}(\mathbf{P}_{\infty}, \epsilon \slash 2)}{\epsilon}.$$

\noindent Finally it is shown that $$\Bigg\| \mathbf{\pi}(0) \mathbf{P} \left( \frac{1}{T} \right) \mathbf{P} \left( \frac{2}{T} \right) \cdots \mathbf{P} \left( \frac{k-1}{T} \right) \mathbf{P} \left( \frac{k}{T} \right) - \mathbf{\pi}\left( \frac{k}{T} \right) \Bigg\|_{TV} \leq \epsilon. $$
\ 

\noindent \underline{Part 2. Assume that $k < t_{mix}\left(\mathbf{P}_{\infty}, \epsilon \slash 2 \right)$} \\ 

\noindent First notice that 
\begin{align*} 
\mathbf{\pi}(0) \mathbf{P} \left( \frac{1}{T} \right) \mathbf{P} \left( \frac{2}{T} \right) &\cdots \mathbf{P} \left( \frac{k-1}{T} \right) \mathbf{P} \left( \frac{k}{T} \right) - \mathbf{\pi}\left( \frac{k}{T} \right) \\
&= \left( \mathbf{\pi}(0) - \mathbf{\pi}\left( \frac{1}{T} \right) \right) \mathbf{P}^{\circ}_{1} + \mathbf{\pi}\left( \frac{1}{T} \right) \mathbf{P}^{\circ}_{2}  - \mathbf{\pi}\left( \frac{k}{T} \right).
\end{align*}

\noindent Repeating this process, it can be shown that
\begin{align*} 
\mathbf{\pi}(0) \mathbf{P} \left( \frac{1}{T} \right) \mathbf{P} \left( \frac{2}{T} \right) &\cdots \mathbf{P} \left( \frac{k-1}{T} \right) \mathbf{P} \left( \frac{k}{T} \right) - \mathbf{\pi}\left( \frac{k}{T} \right) \\
&= \sum_{j=1}^{k} \left( \mathbf{\pi}\left( \frac{j-1}{T} \right)  - \mathbf{\pi}\left( \frac{j}{T} \right) \right) \mathbf{P}^{\circ}_{j}.
\end{align*}

\noindent By the triangle inequality
\begin{align*} 
\Bigg\| \mathbf{\pi}(0) \mathbf{P} \left( \frac{1}{T} \right) \mathbf{P} \left( \frac{2}{T} \right) &\cdots \mathbf{P} \left( \frac{k-1}{T} \right) \mathbf{P} \left( \frac{k}{T} \right) - \mathbf{\pi}\left( \frac{k}{T} \right) \Bigg\|_{TV} \\
&\leq \sum_{j=1}^{k} \Bigg\| \left( \mathbf{\pi}\left( \frac{j-1}{T} \right)  - \mathbf{\pi}\left( \frac{j}{T} \right) \right) \mathbf{P}^{\circ}_{j} \Bigg\|_{TV}.
\end{align*}

\noindent Because $\mathbf{P}^{\circ}_{j}$  is a probability transition matrix  $$\Bigg\| \left( \mathbf{\pi}\left(\frac{j-1}{T} \right)  - \mathbf{\pi}\left( \frac{j}{T} \right) \right) \mathbf{P}^{\circ}_{j} \Bigg\|_{TV} \leq \Bigg\| \mathbf{\pi} \left(\frac{j-1}{T} \right)  - \mathbf{\pi}\left( \frac{j}{T} \right)  \Bigg\|_{TV}.$$  

\noindent This will imply that
\begin{align*} 
\Bigg\| \mathbf{\pi}(0) \mathbf{P} \left( \frac{1}{T} \right) \mathbf{P} \left( \frac{2}{T} \right) &\cdots \mathbf{P} \left( \frac{k-1}{T} \right) \mathbf{P} \left( \frac{k}{T} \right) - \mathbf{\pi}\left( \frac{k}{T} \right) \Bigg\|_{TV} \\
&\leq \sum_{j=1}^{k} \Bigg\| \mathbf{\pi}\left(\frac{j-1}{T}\right)  - \mathbf{\pi}\left( \frac{j}{T} \right) \Bigg\|_{TV}.
\end{align*}

\noindent Using Proposition \ref{pr: one} it is clear that as long as $$ T \geq \frac{3Ln^{3 \slash 2} t_{mix}(\mathbf{P}_{\infty}, \epsilon \slash 2)}{\epsilon \sigma} $$  one has that $$ \Bigg\| \mathbf{\pi} \left( \frac{j-1}{T} \right) - \mathbf{\pi} \left( \frac{j}{T} \right) \Bigg\|_{TV} \leq \frac{\epsilon}{t_{mix}(\mathbf{P}_{\infty}, \epsilon \slash 2)}.$$

\noindent This would imply that 
\begin{align*} 
\Bigg\| \mathbf{\pi}(0) \mathbf{P} \left( \frac{1}{T} \right) \mathbf{P} \left( \frac{2}{T} \right) &\cdots \mathbf{P} \left( \frac{k-1}{T} \right) \mathbf{P} \left( \frac{k}{T} \right) - \mathbf{\pi}\left( \frac{k}{T} \right) \Bigg\|_{TV} \\
&\leq \frac{k \epsilon}{t_{mix}(\mathbf{P}_{\infty}, \epsilon \slash 2)} \\
& \leq \epsilon.
\end{align*}

\noindent Proposition \ref{pr: two} implies that $$T = \frac{3Ln^{3 \slash 2} t_{mix}^{2}(\mathbf{P}_{\infty}, \epsilon \slash 2)}{(1 - 2\sqrt{n} \epsilon)\epsilon} \geq \frac{3Ln^{3 \slash 2} t_{mix}(\mathbf{P}_{\infty}, \epsilon \slash 2)}{\epsilon \sigma}.$$

\noindent This completes our proof.
\end{proof}

\section{\sc {Conclusion}}
\label{sec: conclusion}
Notice that an immediate consequence of Theorem \ref{thm: two} is that there is a tighter asymptotic bound when compared to the previous result in Theorem \ref{thm: one}.  Also convex-combination evolutions are a specific type of continuous evolution, so the class of evolutions is much broader.  The following corollary sums up these two points.

\begin{Cor} \label{cor: one}
Given a time-inhomogeneous, discrete-time Markov chain governed by a continuous evolution in $\mathcal{P}_{n}^{ia}$, for $\epsilon > 0$  and $\mathbf{P}:[0,1] \rightarrow \mathcal{P}_{n}^{ia}$  a continuous function with respect to the matrix norm we have that 

\begin{equation} \label{eq: eleven}
t_{sad}(\mathbf{P}, \epsilon) = O \left( \frac{t_{mix}^{2}(\mathbf{P}_{\infty}, \epsilon \slash 2)}{\epsilon} \right).
\end{equation}
\end{Cor}

A final question that one might have is whether this bound is optimal and the answer is that it is optimal.  To show this it suffices to find one specific function $\mathbf{P}$  such that the stable adiabatic time is exactly a constant  multiplied my the square of the largest mixing time divided by $\epsilon$.  For this one can consider a convex-combination evolution.  Here let $$\mathbf{P}_{0} = \left( \begin{array}{cccc} 1 & 0 & \cdots & 0 \\ 1 & 0 & \cdots & 0 \\ \vdots & \vdots & \ddots & \vdots \\ 1 & 0 & \cdots & 0 \end{array} \right) \qquad \text{ and } \qquad \mathbf{P}_{1} = \left( \begin{array}{cccccc} 0 & 1 & 0 & 0 & \cdots & 0 \\ 0 & 0 & 1 & 0 & \cdots & 0 \\ 0 & 0 & 0 & 1 & \ddots & \vdots \\ \vdots & \vdots & \vdots & \ddots & \ddots & 0 \\ 0 & 0 & 0 & \cdots & 0 & 1 \\ 0 & 0 & 0 & \cdots & 0 & 1 \end{array} \right).$$
\ 

As shown in \cite{bradford2011adiabatic} the adiabatic time for this convex-combination evolution is of asymptotic order of the square of the largest mixing time divided by $\epsilon$.  The only inequality that must hold for the adiabatic time, rather than the stable adiabatic time, is for $\| \mathbf{\pi_{0}} \mathbf{P_{\frac{1}{T}}} \cdots \mathbf{P_{\frac{T}{T}}} - \mathbf{\pi_{\frac{T}{T}}} \|_{TV} < \epsilon$.  Naturally, for all the other inequalities to hold $\| \mathbf{\pi_{0}} \mathbf{P_{\frac{1}{T}}} \cdots \mathbf{P_{\frac{k}{T}}} - \mathbf{\pi_{\frac{k}{T}}} \|_{TV} < \epsilon$  where $1 \leq k < T$  one must select a value of $T$  at least as large as a constant multiplied by the square of the largest mixing time divided by $\epsilon$.  The result in this paper, however, guarantees that this value of $T$  must be of the same asymptotic order.

This shows that the result from Corollary \ref{cor: one} is optimal.  

\section{\sc {Proofs}}
\label{sec: proofs}
\subsection{Proof of Proposition \ref{pr: one}}
\ 

\noindent To begin, consider the creation of an orthonormal basis of eigenvectors associated with $(\mathbb{I} - \mathbf{P}(s))(\mathbb{I} - \mathbf{P}(s))^{T}$  with respect to $\| \cdot \|_{2}$  through a singular value decomposition of $(\mathbb{I} - \mathbf{P}(s))$, where $s \in [0,1]$. \\

\noindent Here let $\sigma_{1}(s) \geq \cdots \geq \sigma_{n-1}(s) = \sigma(s)$  be the positive singular values of $( \mathbb{I} - \mathbf{P}(s))$  with respect to the Euclidean inner product.  This implies that there exists an orthonormal basis $\{ \mathbf{v_{1}}(s), \cdots, \mathbf{v_{n}}(s) \}$  such that $\mathbf{v_{j}}(s)(\mathbb{I} - \mathbf{P}(s))(\mathbb{I} - \mathbf{P}(s))^{T} = \sigma_{j}^{2}(s) \mathbf{v_{j}}(s)$  for $1 \leq j \leq n-1$  and $\mathbf{v_{n}}(s)(\mathbb{I} - \mathbf{P}(s))(\mathbb{I} - \mathbf{P}(s))^{T} = \mathbf{0}$. \\

\noindent Here $\mathbf{v_{n}}(s) = \mathbf{\pi}(s) \slash \| \mathbf{\pi}(s) \|_{2}$.  \\

\noindent To show continuity at  $s $ let $\epsilon > 0$  and first notice that for any $t \in [0,1]$, $(\mathbf{\pi}(t) - \mathbf{\pi}(s))( \mathbb{I} - \mathbf{P}(s) )= \mathbf{\pi}(t)( \mathbf{P}(t) - \mathbf{P}(s))$. \\

\noindent Using the Euclidean norm, it can easily be seen that if $\mathbf{P}(t) \neq \mathbf{P}(s)$  and $t \neq s$, then $$\frac{\| (\mathbf{\pi}(t) - \mathbf{\pi}(s))( \mathbb{I} - \mathbf{P}(s) ) \|_{2}}{\| \mathbf{\pi}(t) - \mathbf{\pi}(s) \|_{2}}=   \frac{\| \mathbf{\pi}(t)( \mathbf{P}(t) - \mathbf{P}(s)) \|_{2}}{ \| \mathbf{\pi}(t) - \mathbf{\pi}(s) \|_{2}}.$$

\noindent Throughout this proof $< \cdot , \cdot >$  we denote the Euclidean inner product. \\

\noindent  For $1 \leq j \leq n$  let $c_{j}(s,t) = <\mathbf{\pi}(t) - \mathbf{\pi}(s), \mathbf{v_{j}}(s)>$.  Then $\mathbf{\pi}(t) - \mathbf{\pi}(s) = \sum_{j=1}^{n} c_{j}(s,t) \mathbf{v_{j}}(s)$.\\

\noindent This will imply that
\begin{align*}
&\frac{\| (\mathbf{\pi}(t) - \mathbf{\pi}(s))(\mathbb{I} - \mathbf{P}(s)) \|_{2}^{2}}{\| \mathbf{\pi}(t) - \mathbf{\pi}(s) \|_{2}^{2}} \\
&\qquad = \frac{<(\mathbf{\pi}(t) - \mathbf{\pi}(s))(\mathbb{I} - \mathbf{P}(s)),(\mathbf{\pi}(t) -\mathbf{\pi}(s))(\mathbb{I} - \mathbf{P}(s))>}{<\mathbf{\pi}(t) - \mathbf{\pi}(s),\mathbf{\pi}(t) - \mathbf{\pi}(s)>} \\
&\qquad = \frac{<\mathbf{\pi}(t) - \mathbf{\pi}(s), (\mathbf{\pi}(t) - \mathbf{\pi}(s))(\mathbb{I} - \mathbf{P}(s))(\mathbb{I} - \mathbf{P}(s))^{T}>}{<\mathbf{\pi}(t) - \mathbf{\pi}(s),\mathbf{\pi}(t) - \mathbf{\pi}(s)>} \\
&\qquad = \frac{<\sum_{j=1}^{n} c_{j}(s,t) \mathbf{v_{j}}(s), \sum_{j=1}^{n-1} \sigma_{j}^{2}(s) c_{j}(s,t) \mathbf{v_{j}}(s)>}{<\sum_{j=1}^{n} c_{j}(s,t) \mathbf{v_{j}}(s),\sum_{j=1}^{n} c_{j}(s,t) \mathbf{v_{j}}(s)>} \\
&\qquad = \frac{\sum_{j=1}^{n-1} \sigma_{j}^{2}(s) c_{j}^{2}(s,t)}{\sum_{j=1}^{n} c_{j}^{2}(s,t)} \\
&\qquad \geq \sigma_{n-1}^{2}(s) \frac{\sum_{j=1}^{n-1} c_{j}^{2}(s,t)}{\sum_{j=1}^{n} c_{j}^{2}(s,t)} \\
&\qquad = \sigma_{n-1}^{2}(s) \left( 1 -\frac{ c_{n}^{2}(s,t)}{\sum_{j=1}^{n} c_{j}^{2}(s,t)} \right) \\
&\qquad = \sigma_{n-1}^{2}(s) \left( 1 - \left( \frac{<\mathbf{\pi}(t) - \mathbf{\pi}(s), \mathbf{v_{n}}(s)>}{\| \mathbf{\pi}(t) - \mathbf{\pi}(s) \|_{2}} \right)^{2} \right).
\end{align*}

\noindent Letting $\mathbf{w}(s,t) = (\mathbf{\pi}(t) - \mathbf{\pi}(s)) \slash \| \mathbf{\pi}(t) - \mathbf{\pi}(s) \|_{2}$, it can be shown that $$ \sigma_{n-1}^{2}(s) \left( 1 - \left( <\mathbf{w}(s,t), \mathbf{v_{n}}(s)> \right)^{2} \right) \leq  \frac{\| \mathbf{\pi}(t)(\mathbf{P}(t) - \mathbf{P}(s)) \|_{2}^{2}}{\| \mathbf{\pi}(t) - \mathbf{\pi}(s) \|_{2}^{2}}.$$

\noindent Because $\mathbf{w}(s,t)$  and $\mathbf{v_{n}}(s)$  are unit vectors, the fact that $$\| \mathbf{w}(s,t) \|_{2}^{2} - 2<\mathbf{w}(s,t),\mathbf{v_{n}}(s)> + \| \mathbf{v_{n}}(s) \|_{2}^{2} = \| \mathbf{w}(s,t) - \mathbf{v_{n}}(s) \|_{2}^{2}$$  

\noindent can be used to show that $$1 - <\mathbf{w}(s,t), \mathbf{v_{n}}(s)> = \frac{1}{2} \| \mathbf{w}(s,t) - \mathbf{v_{n}}(s) \|_{2}^{2}$$  

\noindent and the fact that $$\| \mathbf{w}(s,t) \|_{2}^{2} + 2<\mathbf{w}(s,t),\mathbf{v_{n}}(s)> + \| \mathbf{v_{n}}(s) \|_{2}^{2} = \| \mathbf{w}(s,t) + \mathbf{v_{n}}(s) \|_{2}^{2}$$ 

\noindent can be used to show that $$1 + <\mathbf{w}(s,t), \mathbf{v_{n}}(s)> = \frac{1}{2} \| \mathbf{w}(s,t) + \mathbf{v_{n}}(s) \|_{2}^{2}.$$

\noindent From this it is clear that $1 - \left( <\mathbf{w}(s,t), \mathbf{v_{n}}(s)> \right)^{2} = \| \mathbf{w}(s,t) - \mathbf{v_{n}}(s) \|_{2}^{2} \cdot \| \mathbf{w}(s,t) + \mathbf{v_{n}}(s) \|_{2}^{2} \slash 4.$  Plugging this into the previous equation $$ \frac{\sigma_{n-1}^{2}(s)}{4} \| \mathbf{w}(s,t) - \mathbf{v_{n}}(s) \|_{2}^{2} \cdot \| \mathbf{w}(s,t) + \mathbf{v_{n}}(s) \|_{2}^{2} \leq \frac{\| \mathbf{\pi}(t)(\mathbf{P}(t) - \mathbf{P}(s)) \|_{2}^{2}}{\| \mathbf{\pi}(t) - \mathbf{\pi}(s) \|_{2}^{2}}.$$

\noindent After performing some basic algebra $$ \| \mathbf{\pi}(t) - \mathbf{\pi}(s) \|_{2} \leq \frac{2 \| \mathbf{\pi}(t) (\mathbf{P}(t) - \mathbf{P}(s)) \|_{2}}{\sigma_{n-1}(s) \| \mathbf{w}(s,t) - \mathbf{v_{n}}(s) \|_{2} \cdot \| \mathbf{w}(s,t) + \mathbf{v_{n}}(s) \|_{2}} .$$

\noindent Notice that $<\mathbf{w}(s,t), \mathbf{1}> \slash \sqrt{n} = 0$  and $<\mathbf{v_{n}}(s), \mathbf{1}> \slash \sqrt{n} = 1 \slash \left( \sqrt{n} \| \mathbf{\pi}(s) \|_{2} \right)$  for all $t \in [0,1]$.  Because these are the scalar components of the projections of $\mathbf{w}(s,t)$  and $\mathbf{v_{n}}(s)$  onto $\mathbf{1}$  respectively,it can be shown that the minimum possible value for $\| \mathbf{w}(s,t) - \mathbf{v_{n}}(s) \|_{2}$  and $\| \mathbf{w}(s,t) + \mathbf{v_{n}}(s) \|_{2}$  is at least $1 \slash \left( \sqrt{n} \| \mathbf{\pi}(s) \|_{2} \right).$ \\

\noindent This shows that 
\begin{align*}
\| \mathbf{\pi}(t) - \mathbf{\pi}(s) \|_{2} &\leq \frac{2 n \| \mathbf{\pi}(s) \|_{2}^{2} \cdot  \| \mathbf{\pi}(t) (\mathbf{P}(t) - \mathbf{P}(s)) \|_{2}}{\sigma_{n-1}(s)} \\
& \leq \frac{2 n \| \mathbf{\pi}(t) (\mathbf{P}(t) - \mathbf{P}(s)) \|_{2}}{\sigma_{n-1}(s)} \\
& = \frac{2 n \| \mathbf{\pi}(t) (\mathbf{P}(t) - \mathbf{P}(s)) \|_{2}}{\sigma(s)}.
\end{align*}

\noindent Let $\sigma = \min_{s \in [0,1]} \{ \sigma(s) \}$. \\

\noindent Again for $\mathbf{x},\mathbf{y} \in \mathbb{R}^{n}$ such that $\mathbf{x}$  and $\mathbf{y}$  are probability measures, it is understood that $$ \frac{1}{2} \| \mathbf{x} -\mathbf{y} \|_{2} \leq \| \mathbf{x} - \mathbf{y} \|_{TV} \leq \frac{\sqrt{n}}{2} \| \mathbf{x} - \mathbf{y} \|_{2}.$$

\noindent This will imply that 
\begin{align*}
\| \mathbf{\pi}(t) - \mathbf{\pi}(s) \|_{TV} &\leq \frac{n^{3 \slash 2} \| \mathbf{\pi}(t) ( \mathbf{P}(t) - \mathbf{P}(s) )\|_{1}}{\sigma} \\
&\leq \frac{n^{3 \slash 2} \max_{\nu} \| \mathbf{\nu} ( \mathbf{P}(t) - \mathbf{P}(s) )\|_{1}}{\sigma}
\end{align*}

\noindent where the maximum is taken over all vectors $\mathbf{\nu}$  such that $\| \mathbf{\nu} \|_{1} = 1$. \\

\noindent  Using the matrix norm notation one can conclude that $$ \| \mathbf{\pi}(t) - \mathbf{\pi}(s) \|_{TV} \leq \frac{n^{3 \slash 2} \| \mathbf{P}(t) - \mathbf{P}(s) \|}{\sigma}. $$

\noindent Notice that the space of Lipschitz continuous functions mapping $[0,1]$  to $\mathcal{P}_{n}^{ia}$  are dense in the space of continuous functions mapping $[0,1]$  to $\mathcal{P}_{n}^{ia}$.  This implies that there exists a Lipschitz function $\mathbf{P}^{*}$  with Lipschitz constant $L$  such that $$ \| \mathbf{P}(t) - \mathbf{P}^{*}(t) \| \leq \frac{\sigma \epsilon}{3 n^{3 \slash 2}}$$

\noindent for all $t \in [0,1]$. \\

\noindent One can use the triangle inequality along with this density argument to conclude that
\begin{align*}
\| \mathbf{\pi}(t) - \mathbf{\pi}(s) \|_{TV} &\leq \frac{n^{3 \slash 2} \| \mathbf{P}(t) - \mathbf{P}(s) \|}{\sigma} \\
&= \frac{n^{3 \slash 2}}{\sigma} \left( \| \mathbf{P}^{*}(t) - \mathbf{P}^{*}(s) +  \mathbf{P}(t) - \mathbf{P}^{*}(t) + \mathbf{P}^{*}(s) - \mathbf{P}(s) \| \right) \\
&\leq \frac{n^{3 \slash 2}}{\sigma} \left( \| \mathbf{P}^{*}(t) - \mathbf{P}^{*}(s) \| + \| \mathbf{P}(t) - \mathbf{P}^{*}(t) \| + \| \mathbf{P}^{*}(s) - \mathbf{P}(s) \| \right) \\
&\leq \frac{n^{3 \slash 2}}{\sigma} \left( \| \mathbf{P}^{*}(t) - \mathbf{P}^{*}(s) \| + \frac{\sigma \epsilon}{3 n^{3 \slash 2}} + \frac{\sigma \epsilon}{3 n^{3 \slash 2}} \right) \\
&= \frac{n^{3 \slash 2} \| \mathbf{P}^{*}(t) - \mathbf{P}^{*}(s) \|}{\sigma} + \frac{2 \epsilon}{3}
\end{align*}

\noindent Because $\mathbf{P}^{*}$  is Lipschitz continuous with Lipschitz constant $L$, one has that $\| \mathbf{P}^{*}(t) - \mathbf{P}^{*}(s) \| \leq L \big| t - s \big|$  for all $t, s \in [0,1]$. \\

\noindent This shows that $$\| \mathbf{\pi}(t) - \mathbf{\pi}(s) \|_{TV} \leq \frac{L n^{3 \slash 2} \big| t - s \big|}{\sigma} + \frac{2 \epsilon}{3}.$$

\noindent Clearly if $\epsilon > 0$, then having $$\big| t - s \big| \leq \delta = \frac{ \epsilon \sigma}{3 L n^{3 \slash 2}}$$

\noindent implies $\| \mathbf{\pi}(t) - \mathbf{\pi}(s) \|_{TV} \leq \epsilon$. \\

\noindent This shows that $\mathbf{\pi}$  is continuous at $s \in [0,1]$.  Because one can do this for any $s \in [0,1]$, it is seen that $\mathbf{\pi}$  is continuous with respect to the total variation norm on $[0,1]$.  Because $\delta$  does not depend on the value of $s \in [0,1]$,  it is shown that $\mathbf{\pi}$  is uniformly continuous.

\end{document}